\documentclass[12pt]{amsart} 
\usepackage[english]{babel}
\usepackage{amssymb, amsthm, amsmath}
\usepackage{txfonts}
\usepackage{calrsfs}
\usepackage{mathrsfs}
\usepackage{amscd}
\usepackage{graphicx}
\usepackage[applemac]{inputenc} 
\usepackage[T1]{fontenc} 
\usepackage{hyperref}

\usepackage{cite}

\DeclareMathOperator*{\holim}{holim} 
\DeclareMathOperator*{\hsect}{hsect} 
\DeclareMathOperator*{\psect}{psect}

\DeclareMathOperator*{\Tot}{Tot}

\DeclareMathOperator{\Sing}{\textrm{Sing}_*}

\begin{document}

\pretolerance=5000 

\bibliographystyle{../hsiam}

\setcounter{tocdepth}{1} 
\setlength{\parindent} {0pt}

\newcommand{\cat}[1]{\mathscr{#1}} 
\newcommand{\mcat}[1]{$\mathscr{#1}$} 
\newcommand{\hocat}[1]{Ho\mathscr{#1}} 

\newcommand{\ob}{\textrm{Ob}}
\newcommand{\mor}{\textrm{Mor}}
\newcommand{\id}{\mathbf 1} 

\newcommand{\sSet}{\mathbf{sSet}}
\newcommand{\Alg}{\mathbf{Alg}}
\newcommand{\sPr}{\mathbf{sPr}}
\newcommand{\dgCat}{\mathbf{dgCat}} 
\newcommand{\dgAlg}{\mathbf{dgAlg}} 
\newcommand{\Vect}{\mathbf{Vect}} 
\newcommand{\Ch}{\mathbf{Ch}}
\newcommand{\Chd}{{\mathbf{Ch}_{dg}}} 
\newcommand{\Chdp}{{\mathbf{Ch}_{pe}}} 
\newcommand{\Chp}{{\mathbf{Ch}_{pe}}} 
\newcommand{\uChp}{{\underline {\mathbf{Ch}}{}_{pe}}} 
\newcommand{\skMod}{\mathbf{skMod}} 
\newcommand{\Cat}{\mathbf{Cat}} 
\newcommand{\sCat}{\mathbf{sCat}}
\newcommand{\sModCat}{\mathbf{sModCat}}
\newcommand{\qCat}{\mathbf{qCat}}
\newcommand{\SeCat}{\mathbf{SeCat}}
\newcommand{\staCat}{\mathbf{Cat}^{Ex}_\oo}

\newcommand{\Hom}{\mbox{Hom}}
\newcommand{\uHom}{\mbox{\underline{Hom}}}
\newcommand{\cHom}{\mbox{Hom}^\bullet}

\newcommand{\Aut}{\mbox{Aut}}
\newcommand{\Out}{\mbox{Out}}
\newcommand{\End}{\mbox{End}}

\newcommand{\mods}{\textrm{-Mod}}
\newcommand{\Map}{\mbox{Map}}
\newcommand{\map}{\mbox{map}}
\newcommand{\Tor}{\mbox{Tor}}
\newcommand{\Ext}{\mbox{Ext}}
\newcommand{\Kos}{\mbox{Kos}} 
\newcommand{\Spec}{\mbox{Spec}}

\newcommand{\G}{\mathbb{G}} 

\newcommand{\set}[1]{\mathbb{#1}}
\newcommand{\Q}{\mathbb{Q}}
\newcommand{\C}{\mathbb{C}}
\newcommand{\Z}{\mathbb{Z}}
\newcommand{\R}{\mathbb{R}}

\newcommand{\PR}{\mathbb{P}}
\newcommand{\OO}{\mathscr{O}} 
\newcommand{\MM}{\mathscr{M}} 
\newcommand{\A}{\mathscr{A}} 
\newcommand{\B}{\mathscr{B}} 

\newcommand{\De}{\Delta}
\newcommand{\Ga}{\Gamma}
\newcommand{\Om}{\Omega}
\newcommand{\ep}{\epsilon}
\newcommand{\de}{\delta}
\newcommand{\La}{\Lambda}
\newcommand{\la}{\lambda}
\newcommand{\al}{\alpha}
\newcommand{\om}{\omega}

\newcommand{\ra}{\rightarrow}
\newcommand{\we}{\tilde \ra} 
\newcommand{\cof}{\hookrightarrow}
\newcommand{\fib}{\twoheadrightarrow}
\newcommand{\acof}{\tilde \hookrightarrow}
\newcommand{\afib}{\tilde \twoheadrightarrow}

\newcommand{\oo}{\infty}
\newcommand{\di}{\mbox{d}} 

\newcommand{\op}{^{\textrm{op}}} 
\newcommand{\inv}{^{-1}} 
\newcommand{\nneg}{\tau_{\geq 0}} 

\newcommand{\sgn}{\textrm{sgn }} 
\newcommand{\Bold}{\boldsymbol}
\newcommand{\IF}{\textrm{if }}
\newcommand{\Res}{\textrm{Res}}
\newcommand{\morita}{R\Gamma_\textrm{Morita}(X, \underline k)}
\newcommand{\comment}[1]{}

\theoremstyle{plain}
\newtheorem{thm}{Theorem}
\newtheorem{cor}[thm]{Corollary}
\newtheorem{lemma}[thm]{Lemma}
\newtheorem{propn}[thm]{Proposition}
\newtheorem{conj}{Conjecture}

\makeatletter
\newtheorem*{rep@theorem}{\rep@title}
\newcommand{\newreptheorem}[2]{
\newenvironment{rep#1}[1]{
 \def\rep@title{#2 \ref{##1}}
 \begin{rep@theorem}}
 {\end{rep@theorem}}}
\makeatother

\newtheorem{theorem}{Theorem}
\newreptheorem{theorem}{Theorem}

\theoremstyle{definition}
\newtheorem*{defn}{Definition}
\newtheorem*{altdef}{Alternative Definition}
\newtheorem{eg}{Example}
\newtheorem*{conv}{Convention}
\newtheorem*{notation}{Notation}
\newtheorem*{fact}{Construction}
\newtheorem{qn}{Question}
\newtheorem{fqn}{Further Question}

\theoremstyle{remark}
\newtheorem{rk}{Remark}
\newtheorem{trk}{Temporary Remark}
\newenvironment{frk}
{\pushQED {\qed} \renewcommand {\qedsymbol}{$///$} \trk}
{\popQED \endtrk}
\newtheorem*{claim}{Claim}

\title{Morita Cohomology and Homotopy Locally Constant Sheaves}
\author{Julian V. S. Holstein}
\address{Christ's College and University of Cambridge}
\email{J.V.S.Holstein@dpmms.cam.ac.uk}

\begin{abstract}
We identify Morita cohomology, which is a categorification of the cohmology of a topological space $X$, with the category of homotopy locally constant sheaves of perfect complexes on $X$.
\end{abstract}

\maketitle

\section{Introduction}

In \cite{HolsteinA} \emph{Morita cohomology} $\cat H^M(X)$ was defined as a categorification of \v Cech or singular cohomology of a topological space $X$ with coefficients in a commutative ring $k$. In this paper we make the construction more explicit by identifying Morita cohomology with the category of homotopy locally constant sheaves of perfect chain complexes over $k$ on $X$. 

\begin{reptheorem}{thm-morita-hsect}
Let $X$ have a bounded locally finite good hypercover. 
Then the dg-category $\cat H^{M}(X)$ is quasi-equivalent to 
the dg-category of homotopy locally constant sheaves of perfect complexes.
\end{reptheorem}

The homotopy category of the category of homotopy locally constant sheaves can be considered as the correct derived category of local systems on $X$  in the sense that it contains the abelian category of local systems but its Ext-groups are given by cohomology of $X$ with locally constant coefficients rather than group cohomology of the fundamental group.

The proof proceeds by using a strictification result for diagrams of dg-categories to show $\cat H^{M}(X)$, computed as a homotopy limit, is quasi-equivalent to a category of homotopy cartesian sections of a constant Quillen presheaf. Homotopy cartesian sections are then identified with homotopy locally constant shaves. 

\subsection{Set-up}
We fix throughout this paper a commutative ring $k$ and a topological space $X$ and assume that $X$ has a good hypercover $\mathfrak U = \{U_i\}_{i \in I}$. We say a hypercover is \emph{good} if all connected open sets that occur are contractible.

We will moreover assume that $\mathfrak U$ satisfies the following two conditions, which we sum up by saying $\mathfrak U$ is \emph{bounded locally finite}. 
\begin{itemize}
\item $\mathfrak U$ is locally finite. (Every point has a neighbourhood meeting only finitely many elements of $\mathfrak U$.)
\item There is some positive integer $n$ such that no chain of distinct open sets in $\mathfrak U$ has length greater than $n$. 
\end{itemize}

\begin{rk}
If $X$ is a finite-dimensional CW complex it has a bounded locally finite cover.
One can show this by induction on the $n$-skeleta using
\emph{collaring}, see Lemma 1.1.7 in \cite{Fritsch90}, to extend a bounded locally finite hypercover on $X_{n}$ to one on a neighbourhood of $X_{n}$ in $X_{n+1}$. Then one extends over the $n+1$-cells.
\end{rk}

\subsection{Morita Cohomology}
Morita cohomology $\cat H^{M}(X)$ can be defined as derived global sections of the constant presheaf of dg-categories with fiber equal to the category $\Chp$ of perfect chain complexes over a field $k$. Over an arbitrary ring $k$ it can be defined as $\Chp^{\Sing X}$ using the action of simplicial sets on dg-categories.

Given a good hypercover \mbox{of $X$} one can then compute
$\cat H^{M}(X)$ as the homotopy limit of the constant diagram with fiber $\Chp$ indexed by the hypercover. 

One can also compute it as the homotopy limit of a diagram indexed by the opposite of $I_0 \subset I$, the category of non-degenerate objects of the hypercover: 
$$\cat H^{M}(X) \simeq \holim_{I\op} \Chp \simeq \holim_{I_{0}\op} \Chp$$
In the next section 
we will use strictification to compute this small homotopy limit explicitly as a category of homotopy cartesian sections.
\begin{rk}
For further background and notational conventions see \cite{HolsteinA}. Relations to other work are explained in the introduction of \cite{HolsteinA}. 
\end{rk}

\section{strictification}
\subsection{Background on strictification}\label{sect-strict}

We begin this section with generalities on strictification and the computation of homotopy limits.

Let us consider the fiber $\Ch$ of all chain complexes at first, which has the advantage over $\Chp$ that it is a model category. Model categories are often a convenient model to do computations with $\oo$-categories. However, as the category of model categories is not itself a model category there exist no homotopy limits of model categories. Instead one can compute categories of homotopy cartesian sections and strictification results compare them to homotopy limits of the $\oo$-categories associated with the model categories in question. 

Generally speaking, using strictification to compute a homotopy limit proceeds as follows. Assume we have a localization functor $L\colon  \mathbf{MC} \ra \mathbf{\oo Cat}$ from model categories to some model of \mbox{$(\oo,1)$-categories} and let $\hsect$ denote the category of homotopy cartesian sections of a Quillen presheaf, explained below. Then given a diagram $(\cat M_i)$ of model categories indexed by $I$ one proves $\holim_{i \in I\op} L\cat M_i \simeq L\hsect(I, \cat M_i)$. 

We will proceed by adapting the strictification result for inverse diagrams of simplicial categories from Spitzweck \cite{Spitzweck10} to dg-categories.
$J$ is an \emph{inverse category} if one can associate to every element a non-negative integer, called the \emph{degree}, and every non-identity morphism lowers degree. This is certainly the case for $I_{0}$ if $\mathfrak U$ is bounded locally finite.

We then have to restrict to compact objects in the fibers to compute $R\Ga(X, \uChp)$ rather than $R\Ga(X, \underline \Ch)$. 

\begin{rk}
There is a wide range of strictification results in the literature: For simplicial sets \cite{Dwyer89, Toen05a}, simplicial categories  \cite{Spitzweck10},  Segal categories (Theorem 18.6 of \cite{Simpson01}) and complete Segal spaces \cite{Bergner08, Bergner10}.

Most of the above results make fewer assumptions on the index category, for example Theorem 18.6 of \cite{Simpson01} proves strictification of Segal categories with general Reedy index categories, and a generalization to arbitrary small simplicial index categories is mentioned in Theorem 4.2.1 of \cite{Toen02c}. But since it is unclear to the author how to adapt this proof to the dg-setting and since a bounded locally finite good hypercover for $X$ exists in many cases
we stay with it.
\end{rk}

We will deal with model categories 
that are already enriched in some symmetric monoidal model category $\cat V$ and our $\oo$-categories will be $\cat V$-categories. (Think $\cat V = \sSet$ or $\Ch$.) 

\begin{defn}
Denote by $L$ the localization functor $L\colon  \cat V \mathbf{MC} \ra \mathbf{\cat V Cat}$ that sends $M$ to $M^{cf}$, the subcategory of fibrant cofibrant objects of $M$. 
\end{defn}

The fibrant cofibrant replacement is necessary to ensure that the $\cat V$-hom spaces are invariant under weak equivalences. In the case $\cat V = \sSet$ compare the homotopy equivalence between $LM$ and the Dwyer--Kan localization of $M$.

Let us set up the machinery:
\begin{defn} 
A \emph{left Quillen presheaf} on a small category $I$ is a contravariant functor $M_\bullet\colon  I \ra \mathbf{Cat}$, written as $i \mapsto M_i$ such that for every $i \in \ob(I)$ the category $M_i$ is a model category and for every map $f\colon  i \ra j$ in $I$ the map $f^*\colon  M_j \ra M_i$ is left Quillen. (One can similarly define right Quillen presheaves.)
\end{defn}
\begin{defn} The \emph{constant left Quillen presheaf with fiber $M$}, denoted as $\underline M$ is the Quillen presheaf with $\underline M_i = M$ for all $i$ and $f^* = \id_M$ for all $f$.
\end{defn}
\begin{rk}
One can define Quillen presheaves in terms of pseudofunctors instead of functors, see \cite{Spitzweck10}. One then rectifies the pseudofunctor to turn it into a suitable functor, 
i.e.\ into a left Quillen presheaf as defined above.
\end{rk}

\begin{defn}
Let $M_\bullet$ be a left Quillen presheaf of model categories. 
We define a \emph{left section} to be a tuple consisting of $(X_i, \phi_f)$ for $i \in \ob(I)$ and $f\in \mor(I)$ where $X_i \in M_i$ and $\phi_f\colon  f^*X_j \ra X_i$, satisfies $\phi_g \circ (g^* \phi f) = \phi_{f \circ g}\colon  (f \circ g)^* X_k \ra X_i$ for composable pairs $g\colon  X_i \ra X_j$, $f\colon  X_j \ra X_k$.t

A \emph{morphism of sections} consists of $m_i\colon  X_i \ra Y_i$ in $M_i$ making the obvious diagrams commute. We write the category of sections of $M_\bullet$ as $\psect(I,M_\bullet)$. The levelwise weak equivalences make it into a homotopical category.
\end{defn}
\begin{defn}
A \emph{homotopy cartesian section} is a section for which all the comparison maps $\phi_f\colon  Rf^*X_j \ra X_i$ are isomorphisms in $Ho(M_{i})$. We write the category of homotopy cartesian sections of $M$ as $\hsect(I,M_{\bullet})$.
\end{defn}

If $I$ is an inverse category or $M$ is combinatorial
then the category of left sections $\psect(I,M)$ has an injective model structure, just like a diagram category, in which the weak equivalences and cofibrations are defined levelwise, 
cf. Theorem 1.32 of \cite{Barwick07}. 

We write $L\hsect(I, M_{\bullet})$ for the subcategory of homotopy coherent sections whose objects are moreover fibrant and cofibrant.

Note that $\hsect(I, M_{\bullet})$ is not itself a model category since it is not in general closed under limits.

\begin{rk}
One would like homotopy cartesian sections to be the fibrant cofibrant objects in a suitable model structure.
If we are working with the projective model structure of right sections then (under reasonable conditions) there exists a Bousfield localization, the so-called homotopy limit structure (cf. Theorem 2.44 of \cite{Barwick07}). The objects of $L\hsect_{R}(I,M)$ (which are projective fibrant) are precisely the fibrant cofibrant objects of $(\psect_R)_{holim}(I, M)$. 

The homotopy limit structure on left sections is subtler. It is the subject matter of \cite{Barwick10}. Assuming the category of left sections is a right proper model category Bergner constructs a right Bousfield localization where the cofibrant objects are the homotopy cartesian ones in Theorem 3.2 of \cite{Bergner10}. Without the hard properness assumption the right Bousfield localization only exists as a right semimodel category, cf. \cite{Barwick07a}. 

Note that we will still use model category theory, all we are losing is a conceptually elegant characterization of the subcategory we are interested in.
\end{rk}

\subsection{Strictification for dg-categories}\label{sect-dgstrict}
Our goal now is to prove the following theorem:
\begin{thm}\label{thm-hsect-holim}
Let $I$ be a direct category. Let $M_{i}$ be a presheaf of model categories enriched in $\Ch$. Then $L\hsect(I, M_{\bullet}) \cong \holim_{i \in I\op} LM_{i}$ in $Ho(\dgCat_{DK})$.
\end{thm}
With the results of \cite{HolsteinA} this theorem implies the following: 

\begin{cor}\label{cor-rgamma-hsect}
Let $\{U_i\}_{i \in I}$ be a locally finite good hypercover of $X$. Then $\cat H^{M}(X) \simeq \holim_{I_0\op} \Ch
\simeq L\hsect(I_0, \Ch)$. 
\end{cor}

We will consider in Section \ref{sect-perfect} how to restrict to $\Chp$.

To show Theorem \ref{thm-hsect-holim} we adapt the proof in \cite{Spitzweck10}, replacing enrichments in simplicial sets by enrichment in chain complexes wherever appropriate. For easier reference we write in terms of $\cat V$-categories, where $\cat V = \Ch$ for our purposes and $\cat V = \sSet$ in \cite{Spitzweck10}.

One simplification is that we are assuming the model categories we start with are already enriched in $\Ch$, so that we can use restriction to fibrant cofibrant objects instead of Dwyer--Kan localization as the localization functor.

There are two times two steps to the proof: 
First one defines homotopy embeddings $\rho_{1}$ and $\rho_{2}$ of the two sides into $L\psect(I, \cat V PSh(RLM_\bullet))$. One then shows that their images are given by homotopy cartesian section whose objects are in the image of $M_{i}$.
The first pair of steps are quite formal. The second pair is given by explicit constructions using induction along the degree of the index category.

The proof of the strictification result depends on setting up a comparison between the limit construction and presections. Since the fibrant replacement of $LM_{\bullet}$ is not a Quillen presheaf one has to embed everything into a presheaf of enriched model categories. This is achieved by using the Yoneda embedding.

For the reader's convenience, let us recall the construction of enrichments of presections and presheaves that will be used.

Assume that $\cat V$ is a symmetric monoidal model category and that the we are given a left Quillen presheaf such that all the $M_{i}$ are model $\cat V$-categories. Note that $\cat V$ will be the category $\Ch$ in our application.

If $M_{\bullet}$ is as above and the comparison functors are $\cat V$-functors then $\psect(I, M_\bullet)$ is a model $\cat V$-category:
Tensor and cotensor can be defined levelwise and we define $\uHom_{\psect}(X_\bullet, Y_\bullet)$ as the end $\int_i \uHom(X_i, Y_i)$. 
Since cofibrations and weak equivalences in $\psect(I, M_\bullet)$ are defined levelwise the pushout product axiom holds 
and we have a model $\cat V$-structure.

It follows that the derived internal hom-spaces can be computed as homotopy ends, by cofibrantly and fibrantly replacing source and target: $R\uHom_{\psect}(X_{\bullet}, Y_{\bullet}) = \int_i \uHom((QX)_i, (RY)_i)$. 
See Lemma 2 
in \cite{HolsteinA}. 
In particular if all $M_i$ are dg-model categories then $\psect(I,M)$ is a dg-model category.

\begin{defn}
If $M$ is enriched in $\cat V$ let $\cat VPsh(M)$ be the category of $\cat V$-functors from $M$ to $\cat V$, i.e.\ functors such that the induced map on hom-spaces is a morphism \mbox{in $\cat V$}.
\end{defn}

$\cat V Psh(M)$ is a model category if $\cat V$ has cofibrant hom-spaces or if $\cat V = \Ch$, see Remark 1 in \cite{HolsteinA}. 
Moreover $\cat VPsh(M)$ is enriched, tensored and cotensored over $\cat V$, see for example Chapter 1 of \cite{Kelly82}. 

Next note that there is an enriched Yoneda embedding $M \ra \cat VPsh(M)$. If $\cat V$ has a cofibrant unit and fibrant hom-spaces then the Yoneda embedding factors through the subcategory of fibrant cofibrant objects. 
(To see the image consists of cofibrations, we recall that the maps $0 \ra h^{X} \otimes \id$ are generating cofibrations.)

These conditions are satisfied in $\Ch$.

We write $RLM_{\bullet}$ for $i \mapsto (RLM)_{i}$, where $R$ stands for fibrant replacement in the injective model structure on diagrams of $\cat V$-categories and $L$ is taking fibrant cofibrant objects of every $M_{i}$. 

We now have the following:
\begin{lemma} There is a natural homotopy $\cat V$-embedding 
\[\rho_{1}\colon  L\hsect M_{\bullet} \hookrightarrow L\psect(I, \cat V PSh(RLM_{\bullet}))\] \end{lemma}
\begin{proof}
We have an embedding $\hsect \hookrightarrow \psect$ and homotopy embeddings $M_{i} \hookrightarrow \cat V Psh(RLM_{i})$ which give a homotopy embedding when we apply $L\psect(I, -)$ since the hom-spaces of presections between fibrant cofibrant objects are given by homotopy ends, which are invariant under levelwise weak equivalence.
\end{proof}

\begin{lemma}\label{lemma-limit-embedding}
Let $D_i$ be an $I\op$-diagram of $\cat V$-categories. We have a canonical full $\cat V$-embedding:
\[ \rho_2\colon  \holim D_\bullet = \lim RD_\bullet \hookrightarrow L\psect(I, \cat V PSh(RD_\bullet))\]
\end{lemma} 
\begin{proof} The map to $\psect(I, \cat V Psh(D_{\bullet}^{f}))$ is obtained by composing the Yoneda embedding with the map of $\cat V$-categories $\lim_{i} C_i \ra \psect(I, C_{\bullet})$ that sends $a$ to $\{\pi_i(a)\}$ if $\pi_i\colon  \lim_{j} C_j \ra C_i$ are the universal maps. ($C_{\bullet}$ is not a model category, but we can still take $\psect$ with the obvious meaning, the comparison maps are identities by definition.)
Recall that the hom-space in $\lim C_{i}$ from $\{c_{i}\}$ and $\{d_{i}\}$ is given by $\int_{i} \Hom(c_{i}, d_{i})$.
Hence there is an embedding of the homotopy limit into $\psect(I, C_\bullet)$. 
To show this embedding factors through fibrant cofibrant objects note first that cofibrations are defined levelwise. For fibrations one uses the fibrancy of $RLM_{\bullet}$, this is Lemma 6.3 of \cite{Spitzweck10}.\end{proof}
It follows from this embedding that homotopy equivalences in the homotopy limit are determined levelwise since in $L\psect$ homotopy equivalences are weak equivalences and weak equivalences are defined levelwise. This is Corollary 6.5 in \cite{Spitzweck10}.

From now on we will write $\rho_{2}$ for the case $D_{i} = LM_{i}$.

Next we have to identify the images of $\rho_{1}$ and $\rho_{2}$.
The explicit computation is done in Lemma 6.6 of \cite{Spitzweck10}. The only use of special properties of the category $\sCat$ made in this lemma (and the results needed for it) is the characterization of fibrations in terms of lifting homotopy equivalences. But this characterization is also valid for fibrations in $\dgCat_{DK}$. (A more detailed treatment is available in Section 3.2
of the author's thesis \cite{Holstein0}.)
Thus we have the following results:

\begin{lemma}\label{lemma-image-rho1}
The image of $\rho_{1}$ consists of homotopy cartesian sections $X_{\bullet} \in L\psect(I, \cat V Psh(RLM_{\bullet}))$ such that all $X_{i}$ are in the image of $M_{i}$.
\end{lemma}

\begin{lemma}\label{lemma-image-rho2}
The image of $\rho_{2}$ consists of homotopy cartesian sections $X_{\bullet} \in L\psect(I, \cat V Psh(RLM_{\bullet}))$ such that all $X_{i}$ are in the image of $M_{i}$.
\end{lemma}

Putting this together we obtain a zig-zag of quasi-essentially surjective maps between $L\psect(I, M_{\bullet})$ and $\holim_{I\op} LM$, showing the two categories are isomorphic in $Ho(\dgCat_{DK})$.

\subsection{Restriction to perfect complexes}\label{sect-perfect}
In this section we restrict the equivalence obtained by strictification to sections with compact fibers.

The compact objects in $\Ch$ form the subcategory $\Chp$ consisting of complexes quasi-isomorphic to perfect complexes. 
Note that $\Chp$ is not a model category, so in the next lemma we extend strictification to subcategories.

\begin{propn}\label{propn-perfect-fibers}
The dg-category $\holim_{I\op} \Chp$ is quasi-equivalent to the dg-category $L\hsect(I, \Chp)$, defined to be the subcategory of $L\hsect(I, \Ch)$ consisting of sections $X_{\bullet}$ such that every $X_{i}$ is in $\Chp$.
\end{propn}
\begin{rk} Note that this is not the subcategory of perfect objects in $L\hsect(I, \Ch)$.
\end{rk}
\begin{proof}
Considering the proof of strictification 
we aim to show that $L\hsect(I, \Chp)$ and $\holim_{I\op} \Chp$ can be identified with the subcategory of objects $X_{\bullet} \in L\hsect(I, \cat V Psh(RL\Ch_{\bullet}))$ such that every $X_{i}$ is in the image of $\Chp$.

For $L\hsect(I, \Chp)$ this is immediate from the proof of Lemma \ref{lemma-image-rho1}, as Lemma 6.6 in \cite{Spitzweck10}: 
One
inductively picks $Y'_{i} \simeq X_{i}$ and replaces them by weakly equivalent $Y_{i}$ which form a homtopy Cartesian section. If
$X_{i} \in Im(\Chp)$ then $Y'_{i}$ will also be perfect, ensuring every $Y_{i}$ is the image of a compact object. 

Now we consider Lemma \ref{lemma-image-rho2} and the construction of an object $Y_{\bullet}$ in the homotopy limit that maps to a given homotopy Cartesian section $X_{\bullet}$. 
The proof of the lemma proceeds by lifting $Y_{<n} \in M_{i}(RL\uChp)$ to $Y_{i} \in (RL\uChp)_{i}$ using the quasi-isomorphism $X_{<n} \simeq Y_{<n}$ and the fact that $(RL\uChp)_{i} \ra M_{i}(RL\uChp)$ is a fibration. By assumption $X_{i} \in \uChp$.
But there is a natural map between fibrant diagrams $RL\uChp \ra RL\underline \Ch$ through which $\uChp \ra RL\underline \Ch$ factors by functoriality of fibrant replacement. So from $Y_{<i} \in M_{i}(RL\uChp)$ we can inductively construct $Y_{\bullet}$ such that the $Y_{i}$ live in $(RL\uChp)_{i}$. 
\end{proof}

\begin{thm}\label{thm-holim-hsect} .
$\cat H^M(X) \simeq L\hsect(I_0, \uChp)$ for any locally finite good hypercover $\{U_i\}_{i \in I}$ of $X$ where $I_{0} \subset I$ is the subcategory of non-degenerate objects.
\end{thm}
\begin{proof}

We apply Proposition \ref{propn-perfect-fibers} to Theorem \ref{thm-hsect-holim} and recall Theorem 16 
of \cite{HolsteinA}.
\end{proof}

\section{Homotopy locally constant sheaves}\label{sect-hlc}

Theorem \ref{thm-holim-hsect}
is just a precise way of saying that an object of $\cat H^M(X)$ is given by a collection of chain complexes, one for every open set in the cover, with quasi-isomorphic transition function. We will now turn this into an equivalence with the dg-category of homotopy locally constant hypersheaves of perfect chain complexes. 

To define homotopy locally constant sheaves we put the local model structure (as it is described for example in Section 3.1 
of \cite{HolsteinA}) on presheaves of chain complexes over $k$ on $X$. 
In particular the fibrant objects are exactly objectwise fibrant hypersheaves.

\begin{defn}
We call \emph{homotopy locally constant} a presheaf $\cat F$ such that there is a cover $U_i$ such that all the restrictions $\cat F|_{U_{i}}$ are weakly equivalent to constant sheaves. (In particular the
transition functions between $\cat F(U_i)|_{U_{ij}}$ and $\cat F(U_j)|_{U_{ij}}$ are weak equivalences.)

Then we 
denote by $LC_{H}(X)$ the subcategory of homotopy locally constant hypersheaves of perfect chain complexes. 
This is a dg-category and the hom-spaces are derived hom-spaces of complexes of sheaves. 
\end{defn}

Note that $LC_{H}(X)$ consists of fibrant cofibrant presheaves of chain complexes. It is quasi-equivalent to the category of homotopy locally constant sheaves of perfect complexes on $X$; restricting to fibrant objects simplifies our exposition.

\begin{rk}
The homology sheaves of a homotopy locally constant sheaf are finite dimensional vector bundles which have isomorphisms as transition functions with respect to the above cover, i.e.\ they form local systems. 
\end{rk}

\begin{propn}
Let $X$ be a topological
space with a locally finite good hypercover $\mathfrak U$ and let $I_{0}$ index the nondegnerate connected open sets.
There is a restriction functor from $LC_{H}(X)$ to $L\hsect(I_{0}, \Chp)$ that is quasi-essentially surjective. 
\end{propn}
\begin{proof}
There is an obvious functor $r\colon  LC_{H}(X) \ra L\hsect(I_{0}, \Chp)$ sending a hypersheaf $\cat F$ to $i \mapsto \cat F(U_i)$. (If $\cat F$ is fibrant cofibrant in the local model structure it is fibrant cofibrant in the injective model structure.) We show that $r$ is quasi-essentially surjective by producing a left inverse in the homotopy category. 

Let $\mathfrak U$ also denote the category of all connected 
open sets making up the hypercover $\mathfrak U$. Pick a basis $\mathfrak B$ of contractible sets for the topology of $X$ and assume it is subordinate to $\mathfrak U$ in the sense that any $B \in \mathfrak B$ is contained in any $U \in \mathfrak U$ it intersects. This is possible since $\mathfrak U$ is locally finite.
Consider the presheaf $S^{\mathfrak B}(A)$ on $\mathfrak B$ that sends $B$ to $A_{U}$ where $U$ is minimal containing $B$, such $U$ exists by our assumptions. 
Extend $S^{\mathfrak B}(A)$ to a presheaf $S^{p}(A)$ on $X$ by $S^{p}(A)(W) = \holim_{C \subset W} S^{\mathfrak B}(A)(C)$.
Let $S(A)$ denote a functorial fibrant and cofibrant replacement of $S^{p}(A)$ (in particular it is a sheafification).
Now if we restrict $S^p(A)$ to $U \in \mathfrak U$ there is an obvious weak equivalence with the constant presheaf $A_{U}$, via $S^{p}(B) \simeq A_{U'} \simeq A_{U}$ if $B \subset U' \subset U$. Hence $S(A)$ is a homotopy locally constant sheaf. 

To show that $S(A)(U) \simeq A_{U}$ we can take homology and since the homology sheaves are constant on $U$ the canonical map to the stalk at any point of $U$ is a weak equivalence. (The value at the stalk is weakly equivalent to the limit of the constant diagram $A_{U}$.) 

Hence $r \circ S \simeq \id$ and $r$ is indeed quasi-essentially surjective.
\end{proof}

The following lemma is well-known. We sketch a proof for lack of a reference.

\begin{lemma}\label{lemma-tot-holim}
Let $X_{\bullet}$ be a cosimplicial diagram of chain complexes. Then $\holim X_{\bullet} \simeq \Tot^{\prod} X_{\bullet}$.
\end{lemma}
\begin{proof}
Note that $\Ch$ is an abelian category 
so there is an equivalence of categories between cosimplicial objects in $\Ch$ and nonpositive chain complexes in $\Ch$, we can write this as $\Ch^{\De} \cong \Ch^{\set N}$. Now note that the Dold--Kan correspondence respects levelwise quasi-isomorphisms. (To see the normalized chain functor preserves quasi-isomorphisms consider the splitting of the Moore complex $M(A) = N(A) \oplus D(A)$.) 

It follows that the associated categories with weak equivalences and hence the homotopy categories $Ho(\Ch^{\De})$ and $Ho(\Ch^{\set N})$ are equivalent and the homotopy limit of the cosimplicial diagram is the homotopy limit of the corresponding $\set N$-diagram. 

But taking the homotopy limit of a complex of chain complexes is just taking the product total complex. If the complex is concentrated in two degrees this is the well-known cone construction, which generalizes in the obvious way. \end{proof}
\begin{rk}
It is worth pointing out that while this is an ad-hoc construction there is a complete Dold--Kan theorem for stable $(\oo,1)$-categories in Section 1.2 of \cite{Lurie11}.
\end{rk}

\begin{propn}\label{propn-hom-check}
In the setting of the previous proposition, for $A_\bullet, B_\bullet \in L\hsect(I, \Chp)$ we have: 
\[\uHom_{L\hsect(I, \Chp)}(A_\bullet, B_\bullet) \simeq
\uHom_{LC_{H}(X)}(S(A),S(B))\]
In particular the cohomology groups of $\uHom(A_\bullet, B_\bullet)$ are the Ext groups of $S(A)$ and $S(B)$.
\end{propn}

\begin{proof}
We know that the right hand side can be computed as a \v Cech complex of the good hypercover. 
(See for examples the section Hypercoverings in \cite{Stacks13}.) 
It remains to show that the left-hand side is quasi-isomorphic to $\check C^{*}_{\mathfrak U}(\Hom(S(A), S(B)))$. 
The \v Cech complex is the total complex, 
and hence by Lemma \ref{lemma-tot-holim} the homotopy limit, of the cosimplicial diagram  
\[n \mapsto  \uHom(A_{U_{n}}, B_{U_{n}}) \coloneqq \prod_{i \in I_{n}} \uHom(A_{U_{n}^{i}}, B_{U_{n}^{i}})\] which is in turn equal to the homotopy limit of the diagram $i \mapsto \uHom(A_{U_{n}^{i}}, B_{U_n^{i}})$.
Here we can replace $\uHom_{\cat D}(S(A)(U_{n}^{i}), S(B)(U_{n}^{i}))$ by $\uHom(A_{U_{n}^{i}}, B_{U_{n}^{i}})$ as the $U_{n}^{i}$ are contractible.

By adapting Lemma 2 
of \cite{HolsteinA} to presections one sees that 
the derived functor of $\int \uHom(A_{\bullet}, B_{\bullet})$ is given by $\uHom_{L\hsect}$ between a cofibrant replacement of $A_{\bullet}$ and a fibrant replacement of $B_{\bullet}$ in the injective model category structure. In other words, since all objects in $L\hsect$ are assumed 
 fibrant and cofibrant, $\uHom_{L\hsect}(A_{\bullet}, B_{\bullet})$ is already the derived functor of $\int \uHom(A_{\bullet}, B_{\bullet})$.

Now, adapting Lemma 3.1 of \cite{Spitzweck10} to dg-model categories, we can also compute hom-spaces in $L\hsect$ as a homotopy limit of the diagram $i \mapsto \uHom_{L\hsect(I/i, \uChp)}(A|_{(I/i)}, B|_{(I/i)})$ where the comparison maps are induced by the inclusion of diagrams. 
The underived version follows from a diagram chase comparing the end and the limit of ends, and both sides give fibrant diagrams since $A_{\bullet}$ and $B_{\bullet}$ are fibrant cofibrant. 

By 3.1 of \cite{Spitzweck10} again $\uHom_{L\hsect(I/i, \uChp)}(A|_{(I/i)}, B|_{(I/i)})$ is weakly equivalent to $\uHom_{\Chp}(A_{U_{n}^{i}}, B_{U_{n}^{i}})$. So the objects in the diagrams on the left-hand side and the right-hand side agree.

It remains to show that the comparison maps on the left-hand side correspond to the restriction map of sheaf homs on the right-hand side. Note that giving a sheaf Hom from $S(A)(U)$ to $S(B)(U)$ corresponds to giving morphisms $S(A)(W) \ra S(B)(W)$ for all $W \subset U$, so giving a morphisms of presections in the overcategory $Op(X)\op/U$. But when applying the weak equivalence with $A_{U}$
the only non-identity restrictions 
come from the fixed cover $\mathfrak U$ 
and we can take the limit over $\mathfrak U\op/U$ and obtain the same expression we have on the left-hand side. 

Hence the two homotopy limits agree and the enriched hom-space is weakly equivalent to the \v Cech complex.
\end{proof}

Summing up we have proven: 
\begin{thm}\label{thm-morita-hsect}
Let $X$ be a topological space with a bounded locally finite good hypercover. 
Then $\cat H^M(X)$ is quasi-equivalent to the dg-category $LC_H(X)$.
\end{thm}

The corresponding results also hold if the fiber is $\Ch$.

\begin{rk}
With this interpretation the natural induced map $f^{*}\colon  \cat H^M(Y) \ra \cat H^M(X)$ in Morita cohomology corresponds to the pull-back map of complexes of sheaves.

Since pushforwards of homotopy locally constant sheaves are not homotopy locally constant it is clear that we do not in general expect a map $f_*$ or $f_!$ going in the other direction.
\end{rk}

\begin{rk}
There is an interesting duality between $C^{*}(X)$ and chains on the based loop space $C_{*}(\Om X)$, cf. \cite{Dwyer05}. 
It is well known that $R\uHom_{C_{*}\Om X}(k, k) \simeq C^{*}(X, k)$. We can now  interpret this as saying that the cohomology of $k$ as a $C_{*}(\Om X)$-representation and as a constant sheaf on $X$ agree, and in fact this is a direct consequence 
 of our results characterizing $\cat H^{M}(X)$ as homotopy locally constant sheaves and as $C_{*}(\Om X)$-representations (see \cite{HolsteinA}). 

Conversely if $X$ is simply connected, $k$ is a field and all homology groups are finite dimensional over $k$ it is true that $R\uHom_{{C^{*}(X, k)}}(k,k) \simeq C_{*}(\Om X)$.
It would be interesting to have a similar interpretation of $C^{*}(X)$-modules where it is clear that endomorphisms of $k$ are given by $C_{*}(\Om X)$. 
\end{rk}

\bibliography{../biblibrary}

\end{document}